\documentclass{amsart}
\usepackage[utf8]{inputenc}
\usepackage[T1]{fontenc}
\usepackage{geometry}
\usepackage{amssymb}
\numberwithin{equation}{section}

\def\a{\alpha}

\def\C{\mathbb{C}}
\def\R{\mathbb{R}}
\def\N{\mathbb{N}}

\def\S{{\mathbb{S}}}

\def\M{{\mathbb{M}}}
\def\H{{\mathbb{H}}}

\newtheorem{thm}{Theorem}[section]
\newtheorem{lemma}[thm]{Lemma}

\begin{document}

\title[Duality Theorem]{A DUALITY THEOREM FOR CERTAIN FOCK SPACES}

\author{Jocelyn Gonessa}
\address{Gonessa: Universit\'e de Bangui, D\'epartement 
de math\'ematiques et Informatique, BP.908 Bangui-R\'epublique Centrafricaine}
\email{gonessa.jocelyn@gmail.com}
\email{jocelyn@aims.ac.za}
\subjclass[2000]{Primary  47B35, 32A36, 30H25, 30H30, 46B70, 46M35}

\keywords{Bergman projection, Bergman spaces, Bloch space}
\thanks{Gonessa was supported by \textit{African Institute for Mathematical Sciences (in South Africa) and Agence Universitaire de la Francophonie}}

\begin{abstract}
We characterise functions for the dual spaces of entire functions $f$ such that $fe^{-\phi}\in L^p(\C^n,\rho^{-2}dA)$, $0<p\leq 1$, where $\phi$ is a subharmonic weight and $\rho^{-2}$ is a positive function called under certain conditions regularised version of Laplacian $\Delta\phi$, as described in \cite{C}.  
\end{abstract}

\maketitle

\section{Introduction and main result.}
Let a subharmonic function $\phi$ be given. The spaces we deal with are follows:
$$F^p_\phi =\left\{f\in{\C^n}:\|f\|^p_{F^p_\phi}=\int_{\C^n}|f(z)e^{-\phi(z)}|^p\rho^{-2}(z)dA(z)<\infty \right\},\quad 1\leq p<\infty$$
$$F^\infty_\phi =\left\{f\in{\C^n}:\|f\|_{F^\infty_\phi}=\sup_{z\in\C^n}|f(z)e^{-\phi(z)}|<\infty \right\}$$
where $\rho^{-2}$ is a positive function given. Here $dA$ is the Lebesgue measure on  $\C^n$  normalized so that the volume of the unit ball is equal to one.

Our objective in this work is to prove  for certain $\phi$ and $\rho$ the dual of $F^p_\phi$ is $F^\infty_\phi$, $0<p\leq 1$. More precisely we study the following  particular case. 
\begin{equation}
	\phi(z)=\frac{sN_*(z)-\log|z\bullet z|}{2},\qquad s>0
\end{equation} 
\begin{equation}
	\rho(z)=\sqrt{|z\bullet z|}
\end{equation}
where $$N_*(z)=\sqrt{\|z\|^2+|z\bullet z|}$$ with $\|z\|^2=z\bullet\bar z$ and  
$z\bullet w=z_1w_1+\cdots+z_nw_n$ for all $z=(z_1,\cdots,z_n)$, $w=(w_1,\cdots,w_n)\in\C^n$. 
Namely, $N_*/\sqrt{2}$ is a norm introduced by
Hahn and Pflug, see \cite{HP}. It was shown to be the smallest norm in $\C^n$ that
extends the euclidean norm in $\R^n$ in the following sense. If $N$ is any complex norm in $\C^n$ such that 
$N^2(x)=\|x\|^2=\sum_{j=1}^n x_j^2$ for $x\in\R^n$ and $N(z)=\|z\|$ for $z\in\C^n$, then $N_*(z)/\sqrt{2}\leq N(z)$
for $z\in \C^n$. Moreover, this
norm was shown to be of interest in the study of several problems related to
proper holomorphic mappings and the Bergman kernel, see \cite{G,GY,MY, M1,M2} for example.
For any $s>0$ and $0<p\leq\infty$ we let $L^p_s$ denote the space of Lebesgue mesurable functions $f$ on $\C^n$ 
such that $fe^{-\phi}\in L^p(\C^n,\,\,\rho^{-2}dA)$. In this paper we are going to call $F_s^p$ the Fock spaces $F_\phi^p$, for no particular reason than use notations in \cite{MY}. So we will do the following identifications. $\|\cdot\|_{F^p_\phi}:=\|\cdot\|_{p,s}$ and  $\|\cdot\|_{F^\infty_\phi}:=\|\cdot\|_{\infty,s}$. Let us remark that if $n = 1$ then the spaces $F_s^p$ consist of all entire functions
$f$ such that $f(z)e^{-s|z|^2}\in L^p(\C, |z|^{p-2}dA)$. These are the classical Fock spaces.

The purpose of this paper is to describe the bounded linear functionals
on Fock spaces $F_s^p$ for every $0<p\leq 1$.The answer 
is well known when $1 < p < \infty$ and the problem is solved by using a special pairing, see \cite{G}. The arguments previously provided in \cite{G} and \cite{Z1} play the key role in the present work. However the case $\phi$ is a (nonharmonic) subharmonic function, whose Laplacian satisfies $0<m\leq \Delta \phi(z)\leq M$ ($m$, $M$ positive constants) and $\rho(z)=1$ can be proved (used classical method in \cite{Z1} and Lemma 1 in \cite{OS}). That is the dual of $F_\phi^p$ is $F_\phi^\infty$ for any $0<p\leq 1$.   
 Even better if $\phi$ is a (nonharmonic) subharmonic function, whose $\mu=\Delta \phi$ is a doubling measure and $\rho^{-2}$ is a regularised version of $\mu$, i.e. the positive radius such that $\mu(D(z,\rho(z)))=1$, and $\rho\geq 1$ then  
 the dual of $F_\phi^p$ is $F_\phi^\infty$ for any $0<p\leq 1$ (used classical method in \cite{Z1} and Lemma 19 (a) in \cite{MMO}). Here $D(z,\rho(z))$ is a ball centered in $z$ of radius $\rho(z)$. Our main result is the following.



\newtheorem*{thA}{\bf Theorem A}  
\begin{thA}{\it
Suppose $s>0$ and $0<p\leq 1$. Then the dual space of $F_s^p$ can be identified with $F_s^\infty$.
More precisely, there is a bounded bilinear complex form $L$ on $F_s^p\times F_s^\infty$ such that 
every bounded linear functional on $F_s^p$ has the following form
$$f\mapsto L_g(f):=L(f,g)$$
for some unique $g\in F_s^\infty$. Futhermore the norm of the linear functional on $F_s^p$ 
is equivalent to the norm of $g$ in $F_s^\infty$. Namely, there exits a constant $C$ such that
$$C^{-1}\|g\|_{\infty,s}\leq\|L_g\|\leq C\|g\|_{\infty,s}$$
for all $g\in F_s^\infty$.
}
\end{thA}

The problem of describing the bounded linear functionals on $L^p$, $0 < p \leq 1$, has been studied 
in several papers. In the case of
Hardy spaces, the problem started in \cite{R} and was further pursued in \cite{DRS,F,HM}. For Bergman spaces, 
the problem was studied in \cite{CR,S,Z2,Z3}. Namely,
it was shown that with the classical integral pairing all the classical Fock
spaces have the same dual space when $0 < p\leq 1$. That is the space of the
bounded holomorphic functionnals. In this note we use new tools under a special pairing to prove the same result for
Fock spaces $F_s^p$. 

This work began when I was visiting the
\textit{African Institute for Mathematics Sciences} in Cameroon. I wish to thank the Classical Analysis group managed by Professor David B\'ekoll\'e, for the full discussion.
We achieved the work when I was visiting the \textit{African Institute for Mathematics Sciences} in South Africa.
I wish to thank the mathematics group, and the Directors Professor Mama Foupouagnigni and Professor Barry Green in particular, for very nice visits.

Our starting point are some preliminaries results that we will need in the proof of the main theorem.
\section{Preliminaries}
Let $ n\geq 2$ and consider the nonsingular cone
$$\H:=\{z\in\C^{n+1}: \ z_1^2+\cdots+z_{n+1}^2=0,\,\,z\neq 0\}.$$
This  is the orbit of the vector $(1, i, 0,  \ldots, 0)$ under the $SO(n+1, \C)$-action on $\C^{n+1}.$ 
It is well-known that $\H$ can be identified with  the cotangent bundle of the unit sphere $\S^n$ in 
the $n-$dimensional sphere in $\R^{n+1}$ minus its zero section. It was proved in \cite{OPY} that there 
is a unique (up to a multiplicative constant) $SO(n+1,\,\,\C)-$invariant holomorphic form $\alpha$ on $\H$. 
The restriction of this form to $\H\cap (\C\backslash\{0\})^{n+1}$ is given by
$$\alpha(z)=\sum_{j=1}^{n+1}\frac{(-1)^{j-1}}{z_j}dz_1\wedge\cdots\wedge
\widehat{dz_j}\wedge\cdots\wedge dz_{n+1}.$$

The orthogonal group $O(n + 1,\R)$ acts transitively on the boundary $\M$
of the unit ball in $\H$. Thus there is a unique $O(n + 1,\R)$-invariant probability
measure $\mu$ on $\M$. This measure is induced by the Haar probability
measure of $O(n + 1,\R)$. Following \cite{MY} (see Lemma 2.1 page 506), we have for any
$\C^n$ function $f$ on $\H$ that 
\begin{equation}\label{intMY}
	\int_\H f(z)\a(z)\wedge\bar\a(z)=m_n\int^\infty_0 r^{2n-3}\int_{\partial\M}f(r\xi)d\mu(\xi)dr
\end{equation}
provided that the integrals makes sense. Here
$$m_n=2(n-1)\int_{\{z\in\H:\|z\|<1\}} \a(z)\wedge\overline{\a(z)}$$
For each $s > 0$ and $0 < p<\infty$, let $ L^p_s (\H)$ denote the Lebesgue spaces of
all functions $f$ on $\H$ such as $f\in L^p(\H,w_{s,p})$. Here $w_{s,p}$ is the Gaussian volume form defined on $\H$ as
$$w_{s,p}(z)=\frac{(sp)^{n-1}e^{-sp\|z\|^2/2}}{2^{n-2}m_n(n-2)!},\,\,z\in\H.$$
In the following we adopt some notations. For any $f \in L^p_ s (\H)$ we
write $$\|f\|_{ L^p_s (\H)}=\left(\int_\H |f(z)|^p w_{s,p}(z)\right)^\frac{1}{p},\,\,0<p<\infty$$
and $$\|f\|_{ L^\infty_s (\H)}=\sup_{z\in \H} |f(z)|e^{-s\|z\|^2/2}$$



The weighted Bergman space  $\mathcal{A}^p_s(\H)$ is the closed subspace of $L_s^p(\H)$ consisting of 
holomorphic functions. When $p=2$, the  orthogonal projection $P_s$ from $L_s^2(\H)$ onto $\mathcal{A}^2_s(\H)$ is called 
the weighted Bergman projection.
It is well-known that $P_s$ is the integral operator on $L_s^2(\H)$ given by the formula
$$P_s f(z)=\int_{\H}K_{s,\H}(z,\,\,w)f(w)w_{s,2}(w)$$
where 
$$K_{s,\H}(z, w)=(-1)^{n(n+1)/2}(2i)^n(1+\frac{2s}{n-1}z\bullet\bar w)e^{z\bullet\bar w}$$
 is the reproducing kernel on $\mathcal{A}^2_s(\H)$, see \cite{GY}. This is the weighted Bergman kernel. 
 Let an operator $T_p$ be defined as follows
$$T_pf(z)=C(p)^{1/p}z_{n+1}f(z_1,\cdots,z_n),\qquad z=(z_1,\cdots,z_{n+1})\in\H$$
where $C(p)=\frac{2^{n-3}m_n(n-2)!}{(sp)^{n-1}(n+1)^2}$. As in \cite{MY} (see page 163) the operator $T_p$ will play a key role in our proof. 
\begin{lemma}\label{lemma1}
For any $s>0$ and $0<p<\infty$, the operator $T_p$ is an isometry from $L^p_s(\C^n)$ into $L^p_s(\H)$. 
More precisely, we have
\begin{equation}\label{AshAn0}
	\|T_pf\|_{L^p_s(\H)}=\|f\|_{p,s}
\end{equation}
In addition, the image $\mathcal{E}^p_s(\H)$ of $\mathcal A_s^p(\H)$ under $T_p$ is a closed proper subspace 
of $\mathcal A_s^p(\H)$ and $T_p$ is a unitary operator from $\mathcal A_s^p(\H)$ onto $\mathcal{E}^p_s(\H)$. 
\end{lemma}
The following resut is a crucial ingredient to prove the main lemma of this paper.

\begin{lemma}[See \cite{MMO}]\label{lemmapw1}
	Let an integer $m$. Then for every $R>0$ there exists $A=A(R)$ such that for all $z\in\C^{n}$ $$\sup_{\zeta\in D_n(z)}\left| |\zeta|^2-|z|^2-h_z(\zeta)\right|\leq A$$
	where $h_z$ is a harmonic function in
	$D_n(z)=\{\zeta\in\C^{n}:|\zeta-z|< R\}$ with $h_z(z)=0$.
\end{lemma}



\section{Intermediate results}
In this section we set out an important result of the paper.
\newtheorem*{thB}{\bf Theorem B}  
\begin{thB}\label{theob}
{\it Suppose $s>0$ and $0<p\leq 1$. Then the dual of $\mathcal A_s^p(\H)$ is $\mathcal A_s^\infty(\H)$ under the duality pairing
\begin{equation}\label{equation1}
<f,g>_s=\int_\H f(z)\overline{g(z)}e^{-s\|z\|^2}w_{s,2}(z)
\end{equation}
}
\end{thB}
The starting point of the proof of theorem B is the estimate of the reproducing kernel $\tilde K_{s,\H}$ of $\mathcal E^2_s(\H)$. 
\subsection{Estimate reproducing kernel}
\begin{lemma}\label{lemma2}
Suppose $s>o$ and $0<p\leq 1$. We have
\begin{equation}\label{equation2}
\int_\H |\tilde K_{s,\H}(z,w)|^pw_{s,p}(w)\leq Ce^{sp\|z\|^2/2} 
\end{equation}
where $C$ is a constant.
\end{lemma}
\begin{proof}
If $P_r : \C^{n+1} \to \C^n$ is defined by
$P_r(z_1 ,\cdots, z_n , z_{n+1} ) = (z_1 ,\cdots, z_n )$,
and $F = P r_{|\H}$, then $F : \H \to \C^n\setminus \{0\}$ is a proper holomorphic mapping
of degree 2. We denote by $W$ the branching locus of $F$. The image $F (W )$
of $W$ under $F$ is an analytic subset of $\C^n \setminus\{0\}$. We set $V=F(W)\cup\{0\}$.
The local inverse $\phi$ and $\varphi$ of $F$ are given for $z\in \C^n\setminus V$
by
$$\phi(z)=(z,i\sqrt{z\bullet z})$$
$$\varphi(z)=(z,-i\sqrt{z\bullet z})$$
so that
$$\tilde K_{s,\H}(z,w)=\frac{K_{s,\H}^1(z,w)-K_{s,\H}^2(z,w)}{2}$$
where 
$$K_{s,\H}^1(z,w)=\frac{\bar w_{n+1}}{\overline{\phi_{n+1}(F(w))}}K_{s,\H}(z, \phi(F (w)))$$
$$K_{s,\H}^2(z,w)=\frac{\bar w_{n+1}}{\overline{\phi_{n+1}(F(w))}}K_{s,\H}(z, A(\phi(F (w))))$$
and $A$ is the transformation defined on $\C^{n+1}$ by
$$A(z_1,\cdots,z_{n+1})=(z_1,\cdots,z_n,-z_{n+1}).$$
When $0<p\leq 1$, note that
$$\int_\H |\tilde K_{s,\H}(z,w) |^pw_{s,p}(w)\leq C\int_\H |K_{s,\H}(z,w) |^pw_{s,p}(w).$$
So the desired inequality become
$$\int_\H |K_{s,\H}(z,w) |^pw_{s,p}(w)\leq Ce^{sp\|z\|^2/2}.$$
Let us now denote $I_s$ and $J_s$ the following integrals
$$I_s(z)=\int_\H \left|\left(1+\frac{2s}{n-1}z\bullet\bar w\right)e^{-\frac{s}{2}z\bullet\bar w}\right|^{2p}w_{s,p}(w)$$
and
$$J_s(z)=\int_\H \left|e^{-\frac{s}{2}z\bullet\bar w}\right|^{2p}w_{s,p}(w).$$
In \cite{MY}, Mengotti and Youssfi proved 
\begin{equation}\label{needequa}
\int_\M p_k(\xi)(z\bullet\bar\xi)^ld\mu(\xi)=
\left\{
\begin{array}{ll}
\frac{k!(n-1)!}{(k+n-2)!(2k+n-1)}p_k(z)&if\,\, l=k\\
0 &else

\end{array}
\right.
\end{equation}
where $n$, $k\in\N$ and $p_k$ is a homogeneous polynomial of degree $k$ on $\H$. Then binomial series expansion and (\ref{needequa}) give that
\begin{eqnarray*}
 J_s(z)&=&\int_\H \left|\sum_{k=0}^{+\infty}\frac{(ps)^k}{2^k k!}(z\bullet\bar w)^k
 \right|^2w_{s,p}(w)\\
 &\approx&\sum_{k=0}^{+\infty}\frac{(ps)^{2k}}{(2^k k!)^2}\int_0^{+\infty}r^{2k+2n-3}
 e^{-\frac{psr^2}{2}}dr\int_{\partial\M} |z\bullet\bar\xi|^{2k}d\mu(\xi)\\
  &\approx&\sum_{k=0}^{+\infty}\frac{(ps)^{k-n+1}(n+k-2)!}{2^{k-n+2}(k!)^2}\int_{\partial\M}|z\bullet\bar\xi|^{2k}d\mu(\xi)\\
 &\approx&\sum_{k=0}^{+\infty}\frac{(\frac{ps\|z\|^2}{2})^k}{k!(2k+n-1)}\int_{\partial\M} |z\bullet\bar\xi|^{2k}d\mu(\xi)\\
 &\approx&\frac{e^{\frac{ps\|z\|^2}{2}}}{1+2ps\|z\|^2}.
\end{eqnarray*}
Similary,
\begin{eqnarray*}
 I_s(z)&\leq &C\int_0^{+\infty}\left[ 1+\left(\frac{2sr}{n-1}\right)^p|z\bullet\bar\xi|^p\right]^2
 \left|\sum_{k=0}^{+\infty}({psr}/{2})^{k}
 \frac{(z\bullet\bar\xi)^k}{k!}\right|^2w_{s,p}(w)\\
 &\leq &C\int_0^{+\infty}\left[1+\left(\frac{2sr\|z\|}{n-1}\right)^p\right]^2\sum_{k=0}^{+\infty}({psr}/{2})^{2k}
 ({1}/{k!})^2r^{2k+2n-3}e^{-psr^2/2}dr\int_{\partial\M}|z\bullet\bar\xi|^{2k}d\mu(\xi)\\
  &\leq &C\sum_{k=0}^{+\infty}\frac{({ps}/{2})^{2k}\|z\|^{2k}}{k!(k+n-2)!(2k+n-1)}\int_0^{+\infty}
 (1+\left(\frac{2sr\|z\|}{n-1}\right)^p)^2r^{2k+2n-3}e^{-psr^2/2}dr\\
 &\leq &C\sum_{k=0}^{+\infty}\frac{({ps}/{2})^{2k}\|z\|^{2k}}{k!(k+n-2)!(2k+n-1)}\\
 &&\times\left(\int_0^{\frac{n-1}{2s\|z\|}}
  r^{2k+2n-3}e^{-psr^2/2}dr+2\left(\frac{2sr\|z\|}{n-1}\right)^{2p}\int_{\frac{n-1}{2s\|z\|}}^{+\infty}r^{2k+2n-3}e^{-psr^2/2}dr
  \right)\\
  &\leq &C\sum_{k=0}^{+\infty}\frac{({2}/{s})^{k+n-2}({ps}/{2})^{2k}\|z\|^{2k}}{k!(k+n-2)!(2k+n-1)}\\
  &&\times\left[(k+n-2)!+2\left(\frac{2s\|z\|}{n-1}\right)^{2p}({2}/{s})^p\Gamma(k+p+n-1)\right]\\
  &\leq &C(1+2ps\|z\|^2)e^{ps\|z\|^2/2}.
\end{eqnarray*}
Now using H\"older inequality we obtain that
\begin{eqnarray*}
\int_\H |K_{s,\H}(z,w) |^pw_{s,p}(w)&\leq&\sqrt{I_s(z)J_s(z)}\\
&\leq &Ce^{sp\|z\|^2/2}.
\end{eqnarray*}
This completes the proof of lemma.
\end{proof}
\subsection{Pointwise estimates}
In this section we give the natural growth of functions in $F_s^p$.

\begin{lemma}\label{lemmapw2}
	For any holomorphic function $F$ on $\H$ and $z\in\H$ we have $$|F(z)|\leq C\int_{D_{n+1}(z)\cap\H}|F(\zeta)|\alpha(\zeta)\wedge\overline{\alpha(\zeta)}$$
	where $C$ is a constant independing on $z$. 
\end{lemma}
\begin{proof}
	$F$ being holomorphic then $\Delta F(\zeta)=0$ where $\Delta=\sum_{i=1}^{n+1}\frac{\partial^2}{\partial\zeta_i{\partial\bar\zeta_i}}$.  The divergence theorem implies that
	\begin{align*}
		0=\int_{D_{n+1}(z)\cap\H}\Delta F(\zeta)dA(\zeta)&=\int_{\partial(D_{n+1}(z)\cap\H)}\frac{\partial F}{\partial\nu}(\zeta)d\mu(\zeta)\\
		&=\int_{\partial(D_{n+1}(0)\cap\H)}\frac{\partial F}{\partial r}(z+r\xi)d\mu(\xi)\\
		&=\frac{\partial}{\partial r}\int_{\partial(D_{n+1}(0)\cap\H)}F(z+r\xi)d\mu(\xi)
	\end{align*}
where $\frac{\partial}{\partial\nu}$ is the differentiation in the direction of the external normal. Since the mean value integral at $r=0$ is equal to $F(z)$ then 
\begin{equation}\label{intMY1}
	F(z)=\int_{\partial(D_{n+1}(0)\cap\H)}F(z+r\xi)d\mu(\xi)
\end{equation}
The proof of the lemma arises from (\ref{intMY1}) and (\ref{intMY}). 
\end{proof}

\begin{lemma}\label{lemmapw}
Let $0 < p < \infty$. Then for any holomorphic function $f$ on $\C^n$ we have
$$\|f\|_{\infty,s}\leq C\|f\|_{p,s}$$
where $C$ is a constant. 
\end{lemma}

\begin{proof}
Let $H_z$ a holomorphic function on $\H$ such that $h_z=\Re eH_z$. Then
$$|T_pf(z)e^{-s|z|^2}|^p=|T_pf(z)e^{-sH_z(z)}|^pe^{-sp|z|^2}$$
From Lemma \ref{lemmapw2} we have that
$$
	|T_pf(z)e^{-s|z|^2}|^p\leq C\int_{D_{n+1}(z)\cap\H}|T_pf(w)e^{-sH_z(w)-s|z|^2}|^p\alpha(w)\wedge\overline{\alpha(w)}
$$
Also, from Lemma \ref{lemmapw1} we obtain that.
\begin{equation}\label{AshAn}
	|T_pf(z)e^{-s|z|^2}|^p\leq C\int_{\H}|T_pf(w)e^{-s|w|^2}|^p\alpha(w)\wedge\overline{\alpha(w)}
\end{equation}
for all $z\in\H$.
Finally the lemma arises from (\ref{AshAn}) combined with the estimate (\ref{AshAn0}) at the point $(z,i\sqrt{z\bullet z})\in\H$ where $z\in\C^n$.
\end{proof}
\subsection{Inclusion}
\begin{lemma}\label{lemma4}
Suppose $0<p\leq 1$. Then $\mathcal E_s^p(\H)\subset\mathcal E_s^1(\H)$ and the inclusion is continuous. 
\end{lemma}
\begin{proof}
The starting of the proof is the embedding $F_s^p\subset F_s^1$. Because the desired embedding follows by using 
the isometric $T_p$. So for any $f\in F^p_s$ the Lemma \ref{lemmapw} yields
\begin{eqnarray*}
 \|f\|_{s,1}&\leq&\int_{\C^n}|f(z)|e^{-sN_*^2(z)/2}|z\bullet z|^{-\frac{1}{2}}dA(z)\\
 &\leq&\int_{\C^n}|f(z)|^p|f(z)|^{1-p}e^{-sN_*^2(z)/2}|z\bullet z|^{-\frac{1}{2}}dA(z)\\
 &\leq C&\int_{\C^n}|f(z)|^p(e^{sN_*^2(z)/2}|z\bullet z|^{-\frac{1}{2}}\|f\|_{p,s})^{1-p}
 e^{-sN_*^2(z)/2}|z\bullet z|^{-\frac{1}{2}}dA(z)\\
 &\leq C&\|f\|_{p,s}
\end{eqnarray*}
This proves the desired embedding.
\end{proof}
\section{Proof of the theorem B}
\begin{proof}[Proof of the theorem B]
Consider the bilinear form $L$ defined on $\mathcal E_s^p (H)\times \mathcal E_s^\infty$ (H) by
$$(f, g) \to L g (f ):= L(f, g) =\int_{\H} f (z)\overline{g(z)}w_{2,s}(z).$$ This mapping is well-defined. 
Namely, the Lemma \ref{lemma4} gives that
\begin{eqnarray*}
 |L g (f )|&\leq&\|f\|_{\mathcal E_s^1(\H)}\|g\|_{\mathcal E_s^\infty(\H)}\\
 &\leq&\|f\|_{\mathcal E_s^p(\H)}\|g\|_{\mathcal E_s^\infty(\H)}
\end{eqnarray*}
for all $(f,g)\in \mathcal E_s^p(\H)\times \mathcal E_s^\infty(\H)$. Conversely, if $G$ is a bounded linear
functional on $\mathcal E_s^p (H)$ we must find $g\in\mathcal E^\infty_s (H)$ verifying
$$G(f)=\int_{\H} f (z)\overline{g(z)}w_{2,s} (z)$$
for all $f\in \mathcal E_s^p(H)$. For this goal we choose 
$$g(w)=G(\tilde K_s(w,z)),\,\,w\in\H$$ 
 where $\tilde K_s$ is the Bergman 
kernel of $\mathcal E_s^2(\H)$. Let us prove that $g$ is the desired function. First, we observe that. 
\begin{eqnarray*}
 |g(w)|&\leq&\|G\|\|\tilde K_s(\cdot,w)\|_{\mathcal E_s^p(\H)}\\
 &\leq C&\|G\|e^{s\|w\|^2/2}
\end{eqnarray*}
and thus $\|g\|_{\mathcal E_s^\infty(\H)}\leq C \|G\|$. Second, let us show the following.
\begin{equation}\label{eqqq}
 G(f)=\int_{\H} f (z)\overline{g(z)}w_{2,s} (z)
\end{equation}
for all $f\in \mathcal E_s^p(H)$. To do that we can observe by reproducing property that (\ref{eqqq}) is true for
$f(z)=\tilde K_s(z,w)$. Moreover the set of all finite linear combinations of reproducing kernel functions 
being dense in
$\mathcal E_s^p(H)$ and $\mathcal E_s^p(H) \subset \mathcal E^1_s(H) \subset \mathcal E^2_s(H)$ then (\ref{eqqq}) 
is true. This completes the proof of the theorem B.
\end{proof}

\section{Proof of the main theorem}
\begin{proof}
Consider the bilinear form $L$ defined on $F^p_s\times F^\infty_s$ by
$$(f, g) \to L g (f ):= L(f, g) =\int_{\H} T_p f (z)\overline{Tg(z)}w_{2,s}(z)$$
where $T g(z_1 ,\cdots , z_{n+1}) = z_{n+1} g(z_1 ,\cdots, z_n )$. The functional $L$ is well-defined. Indeed 
Lemma \ref{lemma4} yields 
\begin{eqnarray*}
 |L_g(f)|&\leq C&\|T_pf\|_{\mathcal E^1_s(H)}\|Tg\|_{\mathcal E_s^\infty(H)}\\
 &\leq C&\|T_pf\|_{\mathcal E^p_s(H)}\|Tg\|_{\mathcal E_s^\infty(H)}\\
 &\leq C&\|f\|_{p,s}\|g\|_{p,s}
 \end{eqnarray*}
for all $(f, g)\in F^p_s\times F_s^\infty$. Conversely, if G is a bounded linear functional on $F^p_s$ then 
$G \circ T_p^{-1}$ is in the dual space of $\mathcal E^p_s (H)$.
Hence, from Theorem \ref{theob} there exits $h\in \mathcal E^\infty_s(H)$ such that
$$G \circ T_p^{-1}(\tilde{h})=\int_{\H}\tilde{h}(z)\overline{h(z)}w_{2,s}(z)$$
for all $\tilde{h} \in \mathcal E^p_s(H)$. Finally for $g = T^{-1}h$ we get that.
\begin{eqnarray*}
 G(f)&=&G \circ T_p^{-1}(T_pf)\\
 &=&\int_{\H} T_p f (z)\overline{Tg(z)}w_{2,s}(z)
\end{eqnarray*}
for all $f\in F^p_s$. This completes the proof of the main theorem .
\end{proof}

\bibliographystyle{plain}

\end{document}